\documentclass[12pt,reqno]{amsart}
\usepackage{etoolbox}

\makeatletter
\patchcmd\maketitle
  {\uppercasenonmath\shorttitle}
  {}
  {}{}
\patchcmd\maketitle
  {\@nx\MakeUppercase{\the\toks@}}
  {\the\toks@}
  {}
  {}{}
\patchcmd\@settitle{\uppercasenonmath\@title}{\Large}{}{}
\patchcmd\@setauthors
  {\MakeUppercase{\authors}}
  {\authors}
  {}{}
\makeatother
\usepackage{amsmath,amssymb,amsthm}
\usepackage{color}
\usepackage{url}
\usepackage{tikz-cd}
\usepackage[utf8]{inputenc}
\usepackage[T1]{fontenc}
\newtheorem{theorem}{Theorem}[section]
\newtheorem{definition}{Definition}[section]

\newtheorem{corollary}{Corollary}[section]
\newtheorem{proposition}{Proposition}[section]
\newtheorem{lemma}{Lemma}[section]

\newtheorem{example}{Example}[section]
\numberwithin{equation}{section}
\usepackage[colorlinks=true]{hyperref} 
\hypersetup{urlcolor=blue, citecolor=red , linkcolor= blue}
\usepackage[capitalise,noabbrev,nameinlink]{cleveref}      
\begin{document}
\address{$^{[1]}$ University of Sfax, Tunisia.}
\email{\url{kais.feki@hotmail.com}}


\date{\today}
\author[Kais Feki] {\Large{Kais Feki}$^{1}$}
\title[Some $A$-spectral radius inequalities for $A$-bounded Hilbert space operators]{Some $A$-spectral radius inequalities for $A$-bounded Hilbert space operators}

\maketitle

\begin{abstract}
Let $r_A(T)$ denote the $A$-spectral radius of an operator $T$ which is bounded with respect to the seminorm induced by a positive operator $A$ on a complex Hilbert space $\mathcal{H}$. In this paper, we aim to establish some $A$-spectral radius inequalities for products, sums and commutators of $A$-bounded operators. Moreover, under suitable conditions on $T$ and $A$ we show that
  \begin{equation*}
r_A\left( \sum_{k=0}^{+\infty}c_{k}T^{k}\right)
\leq \sum_{k=0}^{+\infty}|c_{k}|\left[r_A(T)\right]^{k},
  \end{equation*}
where $c_k$ are complex numbers for all $k\in \mathbb{N}$.
\end{abstract}

\section{Introduction and Preliminaries}\label{s1}
Let $\mathcal{B}(\mathcal{H}, \mathcal{K})$ denote the space of all bounded linear operators from a complex Hilbert space $(\mathcal{H}, \langle\cdot\mid \cdot\rangle)$ into a Hilbert space $\mathcal{K}$. We stand $\mathcal{B}(\mathcal{H})$ for $\mathcal{B}(\mathcal{H}, \mathcal{K})$ with $\mathcal{H}=\mathcal{K}$ as a $C^*$-algebra with the operator norm $\|\cdot\|$ and the unit $I$. If $\mathcal{H}=\mathbb{C}^d$, we identify $\mathcal{B}(\mathbb{C}^d)$ with the matrix algebra $\mathbb{M}_d(\mathbb{C})$ of $d\times d$ complex matrices. In all that follows, by an operator we mean a bounded linear operator. The range and the null space of an operator $T$ are denoted by ${\mathcal R}(T)$ and ${\mathcal N}(T)$, respectively. Also, $T^*$ will be denoted to be the adjoint of $T$.

An operator $T$ is called positive if $\langle Tx\mid x\rangle\geq0$ for all $x\in{\mathcal H }$, and we then write $T\geq 0$. The cone of all positive operators will be denoted by $\mathcal{B}(\mathcal{H})^+$.

Throughout this article, we shall assume that $A\in\mathcal{B}(\mathcal{H})$ is a positive operator. Such an $A$ induces the following positive semi-definite sesquilinear form:
$$\langle\cdot\mid\cdot\rangle_{A}:\mathcal{H}\times \mathcal{H}\longrightarrow\mathbb{C},\;(x,y)\longmapsto\langle x\mid y\rangle_{A} :=\langle Ax\mid y\rangle.$$
Notice that the induced seminorm is given by $\|x\|_A=\langle x\mid x\rangle_A^{1/2}=\|A^{1/2}x\|$, for every $x\in \mathcal{H}$. Here, $A^{1/2}$ is denoted to be the square root of $A$. This makes $\mathcal{H}$ into a semi-Hilbertian space. One can check that $\|\cdot\|_A$ is a norm on $\mathcal{H}$ if and only if $A$ is injective, and that $(\mathcal{H},\|\cdot\|_A)$ is complete if and only if $\mathcal{R}(A)$ is closed.

The following celebrated assertion is known as the Douglas theorem or Douglas majorization theorem.
\begin{theorem}\label{doug}(\cite[Theorem 1]{doug})
If $T, S \in \mathcal{B}(\mathcal{H})$, then the following statements are equivalent:
\begin{enumerate}
\item[{\rm (i)}] $\mathcal{R}(S) \subseteq \mathcal{R}(T)$;
\item[{\rm (ii)}] $TD=S$ for some $D\in \mathcal{B}(\mathcal{H})$;
\item[{\rm (iii)}]$SS^* \leq \lambda^2 TT^*$ for some $\lambda\geq 0$ (or equivalently $\|S^*x\| \leq \lambda\|T^*x\|$ for all $x\in \mathcal{H}$).
\end{enumerate}
Moreover, if one of these conditions holds then, there exists a unique solution $Q\in\mathcal{B}(\mathcal{H})$ of the equation $TX=S$ (known as the Douglas solution) so that
\begin{enumerate}
\item[{\rm (a)}] $\|Q\|^2=\inf\left\{\mu\,;\;SS^*\leq \mu TT^*\right\}$,
\item[{\rm (b)}] $\mathcal{N}(Q)=\mathcal{N}(S)$,
\item[{\rm (c)}] $\mathcal{R}(Q) \subseteq \overline{\mathcal{R}(T^{*})}$.
\end{enumerate}
\end{theorem}
\begin{definition} (\cite{acg1})
Let $A\in\mathcal{B}(\mathcal{H})^+$ and $T \in \mathcal{B}(\mathcal{H})$. An operator $S\in\mathcal{B}(\mathcal{H})$ is called an $A$-adjoint of $T$ if for every $x,y\in \mathcal{H}$, the identity $\langle Tx\mid y\rangle_A=\langle x\mid Sy\rangle_A$ holds. That is $S$ is solution in $\mathcal{B}(\mathcal{H})$ of the equation $AX=T^*A$.
\end{definition}
The set of all operators which admit $A^{1/2}$-adjoints is denoted by $\mathcal{B}_{A^{1/2}}(\mathcal{H})$. By applying Theorem \ref{doug}, it can observed that
\begin{equation}\label{abbbbbbbb}
\mathcal{B}_{A^{1/2}}(\mathcal{H})=\left\{T \in \mathcal{B}(\mathcal{H})\,;\;\exists \,\lambda > 0\,;\;\|Tx\|_{A} \leq \lambda \|x\|_{A},\;\forall\,x\in \mathcal{H}  \right\}.
\end{equation}
Operators in $\mathcal{B}_{A^{1/2}}(\mathcal{H})$ are called $A$-bounded. Note that $\mathcal{B}_{A}(\mathcal{H})$ is a subalgebra of $\mathcal{B}(\mathcal{H})$ which is neither closed nor dense in $\mathcal{B}(\mathcal{H})$ (see \cite{acg1}). Further, clearly $\langle\cdot\mid\cdot\rangle_{A}$ induces the following seminorm on $\mathcal{B}_{A^{1/2}}(\mathcal{H})$:
\begin{equation}\label{aseminorm}
\|T\|_A:=\sup_{\substack{x\in \overline{\mathcal{R}(A)},\\ x\not=0}}\frac{\|Tx\|_A}{\|x\|_A}=\sup\left\{\|Tx\|_{A}\,;\;x\in \mathcal{H},\,\|x\|_{A}= 1\right\}<\infty.
\end{equation}
In addition, it was shown in \cite{fg} that for $T\in\mathcal{B}_{A^{1/2}}(\mathcal{H})$ we have:
\begin{equation}\label{semi2}
\|T\|_A=\sup\left\{|\langle Tx\mid y\rangle_A|\,;\;x,y\in \mathcal{H},\,\|x\|_{A}=\|y\|_{A}= 1\right\}.
\end{equation}
We would like to mention that the inclusion $\mathcal{B}_{A^{1/2}}(\mathcal{H})\subseteq \mathcal{B}(\mathcal{H})$ is in general strict as it is shown in the following example.
 \begin{example}
Let $\mathcal{H}=\ell_{\mathbb{N}^*}^2(\mathbb{C})$ and $A$ be the diagonal operator on $\ell_{\mathbb{N}^*}^2(\mathbb{C})$ defined as $Ae_n=\frac{e_n}{n!}$ for all $n\in \mathbb{N}^*$, where $(e_n)_{n\in \mathbb{N}^*}$ is denoted to be the canonical basis of $\ell_{\mathbb{N}^*}^2(\mathbb{C})$. Let also $T_\ell$ be the backward shift operator on $\ell_{\mathbb{N}^*}^2(\mathbb{C})$ (that is $T_\ell e_1=0$ and $T_\ell e_n=e_{n-1}$ for all $n\geq 2$). It can observed that $\|e_n\|_A=\frac{1}{\sqrt{n!}}$ for all $n\in \mathbb{N}^*$ and $\|T_\ell e_n\|_A=\frac{1}{\sqrt{(n-1)!}}=\sqrt{n}\|e_n\|_A$ for $n\geq 2$. Hence, we infer that $\|T_\ell\|_A=+\infty$ and $T_\ell\in \mathcal{B}(\ell_{\mathbb{N}^*}^2(\mathbb{C}))\setminus\mathcal{B}_{A^{1/2}}(\ell_{\mathbb{N}^*}^2(\mathbb{C}))$.
\end{example}
Before we move on, let us emphasize the following two facts. If $T\in \mathcal{B}_{A^{1/2}}(\mathcal{H})$, then
\begin{equation}\label{semiiineq}
\|Tx\|_A\leq \|T\|_A\|x\|_A,\;\forall\,x\in \mathcal{H}.
\end{equation}
Moreover for every $T,S\in \mathcal{B}_{A^{1/2}}(\mathcal{H})$, we have
\begin{equation}\label{semiiineq2}
\|TS\|_A\leq \|T\|_A\|S\|_A.
\end{equation}
Recently, the present author introduced in \cite{feki01} the concept of the $A$-spectral radius of $A$-bounded operators. Henceforth, $A$ is implicitly understood as a positive operator. His definition reads as follows: for $T\in \mathcal{B}_{A^{1/2}}(\mathcal{H})$ we have
\begin{equation}\label{newrad}
r_A(T):=\displaystyle\inf_{n\in \mathbb{N}^*}\|T^n\|_A^{\frac{1}{n}}=\displaystyle\lim_{n\to\infty}\|T^n\|_A^{\frac{1}{n}}.
\end{equation}
Notice that the second equality in \eqref{newrad} is also proved in \cite{feki01}. If $A=I$, we get the well-known spectral radius formula of an operator denoted simply by $r(T)$. The study of the spectral radius of Hilbert space operators received considerable attention in the last decades. The reader may consult \cite{Ki,A.K.2,BBP,dfilot,bakfeki02} and the references therein.

In the next proposition we collect some properties of the $A$-spectral radius.
\begin{proposition}(\cite{feki01})\label{22}
Let $T,S\in \mathcal{B}_{A^{1/2}}(\mathcal{H})$. Then the following assertions hold:
\begin{itemize}
  \item [(1)] If $TS=ST$, then $r_A(TS)\leq r_A(T)r_A(S)$.
  \item [(2)] If $TS=ST$, then $r_A(T+S)\leq r_A(T)+r_A(S)$.
 \item [(3)] $r_A(T^k)=[r_A(T)]^k$ for all $k\in \mathbb{N}^*$.
\end{itemize}
\end{proposition}
It should be emphasized that $r_A(\cdot)$ satisfies the commutativity property, which asserts that
\begin{equation}\label{commut}
r_A(TS)=r_A(ST),
\end{equation}
for every $T,S\in \mathcal{B}_{A^{1/2}}(\mathcal{H})$ (see \cite{feki01}).

Recently, the $A$-numerical range of $T\in \mathcal{B}(\mathcal{H})$ is introduced by H. Baklouti et al. in \cite{bakfeki01} as
$W_A(T) = \big\{{\langle Tx\mid x\rangle}_A: \,\, x\in \mathcal{H},\, {\|x\|}_A = 1\big\}$.
This new concept is a nonempty convex subset of $\mathbb{C}$ which is not necessarily closed. Moreover its supremum modulus is called the $A$-numerical radius of $T$ and it is given by
$$\omega_A(T) = \sup\big\{|\langle Tx\mid x\rangle_A|: \,\, x\in \mathcal{H},\, \|x\|_A = 1\big\}.$$
Notice that it may happen that $\omega_A(T) = + \infty$ for some $T\in\mathcal{B}(\mathcal{H})$. Indeed, one can consider the following operators
$A = \begin{pmatrix}
1 & 0 \\
0 & 0
\end{pmatrix}\in \mathbb{M}_2(\mathbb{C})^+$ and
$T = \begin{pmatrix}
0 & 1 \\
1 & 0
\end{pmatrix}\in \mathbb{M}_2(\mathbb{C})$. However, $\omega_A(T)< + \infty$ for every $A$-bounded operator $T$. More precisely, for all $T\in\mathcal{B}_{A^{1/2}}(\mathcal{H})$ we have
\begin{equation}\label{equivalentsemi}
\tfrac{1}{2}\|T\|_A \leq \omega_A(T)\leq \|T\|_A.
\end{equation}
On the other hand, the present author proved in \cite{feki01} that for every $T\in\mathcal{B}_{A^{1/2}}(\mathcal{H})$ we have
\begin{equation}\label{dom}
r_A(T) \leq \omega_A(T).
\end{equation}

Now, we mention that neither the existence nor the uniqueness of an $A$-adjoint operator is guaranteed. The set of all operators which admit $A$-adjoints is denoted by $\mathcal{B}_{A}(\mathcal{H})$. By applying Theorem \ref{doug} we see that
$$\mathcal{B}_{A}(\mathcal{H})=\left\{T\in \mathcal{B}(\mathcal{H})\,;\;\mathcal{R}(T^{*}A)\subset \mathcal{R}(A)\right\}.$$
Like $\mathcal{B}_{A^{1/2}}(\mathcal{H})$, the subspace $\mathcal{B}_{A}(\mathcal{H})$ is a subalgebra of $\mathcal{B}(\mathcal{H})$ which is neither closed nor dense in $\mathcal{B}(\mathcal{H})$. In addition, we have $\mathcal{B}_{A}(\mathcal{H}) \subseteq \mathcal{B}_{A^{1/2}}(\mathcal{H}))$ (see \cite[Proposition 1.2.]{acg3}).

Let $T\in\mathcal{B}_{A}(\mathcal{H})$. The Douglas solution of the equation $AX = T^*A$ is a distinguished $A$-adjoint operator of $T$, which is denoted by $T^{\sharp_A}$. Note that, $T^{\sharp_A} = A^{\dag}T^*A$ in which $A^{\dag}$ is denoted to be the Moore-Penrose inverse of $A$. It is important to mention that if $T\in\mathcal{B}_{A}(\mathcal{H})$, then $T^{\sharp_A}\in\mathcal{B}_{A}(\mathcal{H})$, $\|T^{\sharp_A}\|_A=\|T\|_A$ and $(T^{\sharp_A})^{\sharp_A} = P_ATP_A$. Here, $P_A$ denotes the orthogonal projection onto $\overline{\mathcal{R}(A)}$. Furthermore, if $T, S\in\mathcal{B}_{A}(\mathcal{H})$, then $(TS)^{\sharp_A} = S^{\sharp_A}T^{\sharp_A}$. In addition, an operator $U\in  \mathcal{B}_A(\mathcal{H})$ is said to be $A$-unitary if $U^{\sharp_A} U=(U^{\sharp_A})^{\sharp_A} U^{\sharp_A}=P_A$. For proofs and more facts about this class of operators, the reader is invited to consult \cite{acg1,acg2,bakfeki01,bakfeki04} and their references.

Recently, many results covering some classes of operators on a complex Hilbert space $\big(\mathcal{H}, \langle \cdot\mid \cdot\rangle\big)$
are extended to $\big(\mathcal{H}, {\langle \cdot\mid \cdot\rangle}_A\big)$ (see, e.g., \cite{zamani2,bakfeki01,bakfeki04,zamani1,majsecesuci}).

The remainder of this paper is organized as follows. Section \ref{s2} is meant to establish several results governing $r_A(\cdot)$. Some of the obtained results will be a natural generalization of the well-known case $A=I$ and extend the works of F. Kittaneh et al. \cite{Ki,A.K.2,BBP,dfilot}.

In section \ref{s3} we consider the power series $f\left( z\right) =\sum_{n=0}^{\infty }c_{n}z^{n}$ with complex coefficients and $f_{c}\left( z\right) :=\sum_{n=0}^{\infty }\left\vert c_{n}\right\vert z^{n}$. Obviously, $f$ and $f_{c}$ have the same radius of convergence and if $c_{n}\geq 0,$ for all $n\in \mathbb{N}^*$, then $f_{c}=f$. The main target of this section in to establish, under some conditions on $A$ and $T$, a relation between $r_A[f(T)]$ and $f_c[r_A(T)]$. The obtained results cover the work of S. S. Dragomir \cite{dfilot}.

\section{$A$-spectral radius inequalities}\label{s2}
In this section, we will prove several inequalities related to $r_A(T)$ when $T$ is an $A$-bounded operator. In all what follows, we consider the Hilbert space $\mathbb{H}=\oplus_{i=1}^d\mathcal{H}$ equipped with the following inner-product:
$$\langle x, y\rangle=\sum_{k=1}^d\langle x_k\mid y_k\rangle,$$
 for all $x=(x_1,\cdots,x_d)\in \mathbb{H}$ and $y=(y_1,\cdots,y_d)\in \mathbb{H}$. Let $\mathbb{A}$ be a ${d\times d}$ operator diagonal matrix with diagonal entries are the positive operator $A$, i.e.
 \begin{equation*}
\mathbb{A}=\begin{pmatrix}A & 0 &\cdots& 0\\
0& A &\cdots& 0\\
\vdots & \vdots & \vdots & \vdots\\
0 & 0 &\cdots& A
\end{pmatrix}.
\end{equation*}
Clearly, $\mathbb{A}\in \mathcal{B}(\mathbb{H})^+$. So, the semi-inner product induced by $\mathbb{A}$ is given by
$$\langle x, y\rangle_{\mathbb{A}}= \langle \mathbb{A}x, y\rangle=\sum_{k=1}^d\langle Ax_k\mid y_k\rangle=\sum_{k=1}^d\langle x_k\mid y_k\rangle_A,$$
 for all $x=(x_1,\cdots,x_d)\in \mathbb{H}$ and $y=(y_1,\cdots,y_d)\in \mathbb{H}$.

In order to prove our first main result in this section, we need the following lemma.
\begin{lemma}\label{lemmajdid}
Let $\mathbb{T}= (T_{ij})_{d \times d}$ be such that $T_{ij}\in \mathcal{B}_{A^{1/2}}(\mathcal{H})$ for all $i,j$. Then, $\mathbb{T}\in\mathcal{B}_{\mathbb{A}^{1/2}}(\mathbb{H})$. Moreover, we have
\begin{equation}\label{tag0}
\|\mathbb{T}\|_{\mathbb{A}}\leq \|\widehat{\mathbb{T}}^{\mathbb{A}}\|,
\end{equation}
where $\widehat{\mathbb{T}}^{\mathbb{A}}=(\|T_{ij} \|_A)_{d\times d}$.
\end{lemma}
\begin{proof}
 Let $x=(x_1,\cdots,x_d)\in \mathbb{H}$. It can be seen that
\begin{align}\label{jdjjj}
\|\mathbb{T}x\|_{\mathbb{A}}^2
&= \|\mathbb{A}^{1/2}\mathbb{T}x\|^2\nonumber\\
&= \sum_{k=1}^d\left\|\sum_{j=1}^dT_{kj}x_j\right\|_A^2\nonumber\\
&\leq\sum_{k=1}^d\left(\sum_{j=1}^d\|T_{kj}x_j\|_A\right)^2.
\end{align}
On the other hand, since $T_{ij}\in \mathcal{B}_{A^{1/2}}(\mathcal{H})$ for all $i,j$, then by \eqref{abbbbbbbb} there exists $\lambda_{ij} > 0$ such that
\begin{equation}\label{loula}
\|T_{ij}x\|_A\leq \lambda_{ij} \left\|x\right\|_A,
\end{equation}
for all $x\in \mathcal{H}$ and $i,j\in\{1,\cdots d\}$. So, by using \eqref{loula} and the Cauchy-Shwarz inequality we get
\begin{align*}
\|\mathbb{T}x\|_{\mathbb{A}}^2
&\leq\sum_{k=1}^d\left(\sum_{j=1}^d\lambda_{kj} \left\|x_j\right\|_A\right)^2\\
&\leq d(\max_{k,j}\lambda_{kj})\left(\sum_{j=1}^d \left\|x_j\right\|_A\right)^2\\
&\leq d^2(\max_{k,j}\lambda_{kj})\sum_{j=1}^d \left\|x_j\right\|_A^2=\lambda^2\|x\|_\mathbb{A}^2,
\end{align*}
where $\lambda:=d\sqrt{\max_{k,j}\lambda_{kj}}$. Hence, by \eqref{abbbbbbbb}, we infer that $\mathbb{T}\in\mathcal{B}_{\mathbb{A}^{1/2}}(\mathbb{H})$. So, by using \eqref{aseminorm}, we have
$$\|\mathbb{T}\|_{\mathbb{A}}=\sup\{\|\mathbb{T}x\|_{\mathbb{A}},\;x\in \mathbb{H}, \|x\|_{\mathbb{A}}=1\}.$$
In order to prove \eqref{tag0}, it suffices to show that
\begin{equation}\label{tttag0}
\|\mathbb{T}x\|_{\mathbb{A}}\leq \|\widehat{\mathbb{T}}^{\mathbb{A}}\|\, \|x\|_{\mathbb{A}},\;\forall\,x\in \mathbb{H}.
\end{equation}
Let $x=(x_1,\cdots,x_d)\in \mathbb{H}$. Let $\widehat{x}^{A}$ denote $(\|x_1\|_A,\cdots,\|x_d\|_A) \in \Bbb R^d$. Notice that $\|\widehat{x}^{A}\|_{\mathbb{A}} = \|x\|_{\mathbb{A}}$. By using \eqref{semiiineq} and \eqref{jdjjj}, one can see that
\begin{align*}
\|\mathbb{T}x\|_{\mathbb{A}}^2
&\leq\sum_{k=1}^d\left(\sum_{j=1}^d\|T_{kj}x_j\|_A\right)^2\\
&\leq\sum_{k=1}^d\left(\sum_{j=1}^d\|T_{kj}\|_A\|x_j\|_A\right)^2\\
&=\left\|\widehat{\mathbb{T}}^{\mathbb{A}} \widehat{x}^{A}\right\|^2 \\
&\leq \|\widehat{\mathbb{T}}^{\mathbb{A}}\|^2\|\widehat{x}^{A}\|^2 = \|\widehat{\mathbb{T}}^{\mathbb{A}}\|^2 \| x\|_{\mathbb{A}}^2.
\end{align*}
Hence, \eqref{tttag0} is proved and thus the proof is complete.
\end{proof}

Now, we are in a position to prove our first main result in this section.
\begin{theorem}\label{main222222}
Let $\mathbb{T}= (T_{ij})_{d \times d}$ be a $d \times d$ operator matrix be such that $T_{ij}\in \mathcal{B}_{A^{1/2}}(\mathcal{H})$ for all $i,j$ and $\widehat{\mathbb{T}}^{\mathbb{A}}=(\|T_{ij} \|_A)_{d\times d}$. Then, $r_{\mathbb{A}}(\mathbb{T})\leq r(\widehat{\mathbb{T}}^{\mathbb{A}}).$ That is
\begin{equation}\label{kk2020}
r_{\mathbb{A}}\left[ \begin{pmatrix}
T_{11} & T_{12} &\cdots& T_{1d}\\
T_{21} & T_{22} &\cdots& T_{2d}\\
\vdots & \vdots & \vdots & \vdots\\
T_{d1} & T_{d2} &\cdots& T_{dd}\\
\end{pmatrix}\right] \leq
r \left[\begin{pmatrix}
\|T_{11}\|_A & \|T_{12}\|_A &\cdots& \|T_{1d}\|_A\\
\|T_{21}\|_A & \|T_{22}\|_A &\cdots& \|T_{2d}\|_A\\
\vdots & \vdots & \vdots & \vdots\\
\|T_{d1}\|_A & \|T_{d2}\|_A &\cdots& \|T_{dd}\|_A\\
\end{pmatrix} \right].
\end{equation}
\end{theorem}
\begin{proof}
Notice, in general, that for operators $\mathbb{T}= (T_{ij})_{d \times d}$ and $\mathbb{S}= (S_{ij})_{d \times d}$ such that $T_{ij}\in \mathcal{B}_{A^{1/2}}(\mathcal{H})$ for all $i,j$ and $S_{ij}\in \mathcal{B}_{A^{1/2}}(\mathcal{H})$ for all $i,j$ respectively, we have
\begin{equation}\label{tag00001}
\|\widehat{\mathbb{T}\mathbb{S}}^{\mathbb{A}}\|\leq \|\widehat{\mathbb{T}}^{\mathbb{A}}\|\|\widehat{\mathbb{S}}^{\mathbb{A}}\|.
\end{equation}
Indeed, we have
$$
\left(\widehat{\mathbb{T}\mathbb{S}}^{\mathbb{A}}\right)_{kj}=\left\|\sum_{\ell=1}^d T_{k\ell} S_{\ell j} \right\|_A\leq \sum_{\ell=1}^d \|T_{k\ell} \|_A\,\|S_{\ell j}\|_A=\left(\widehat{\mathbb{T}}^{\mathbb{A}}\widehat{\mathbb{S}}^{\mathbb{A}}\right)_{kj},
$$
for all $k,j$. Therefore, by the norm monotonicity of matrices with nonnegative entries, we see that
$$\|\widehat{\mathbb{T}\mathbb{S}}^{\mathbb{A}}\|\leq \|\widehat{\mathbb{T}}^{\mathbb{A}}\widehat{\mathbb{S}}^{\mathbb{A}}\|.$$
This shows \eqref{tag00001} since $\|\widehat{\mathbb{T}}^{\mathbb{A}}\widehat{\mathbb{S}}^{\mathbb{A}}\|\leq \|\widehat{\mathbb{T}}^{\mathbb{A}}\|\,\|\widehat{\mathbb{S}}^{\mathbb{A}}\|$.

Now, by using \eqref{tag0} together with \eqref{tag00001} and an induction argument, we get
 \begin{equation*}
\|\mathbb{T}^n\|_\mathbb{A}\leq \left\|\widehat{\mathbb{T}^n}^\mathbb{A}\right\|\leq \left\|\left(\widehat{\mathbb{T}}^\mathbb{A} \right)^n\right\|,
 \end{equation*}
for all $n\in\mathbb N^*$. Thus, by using \eqref{newrad} we obtain
$$r_{\mathbb{A}}(\mathbb{T})=\lim_{n\to \infty}\|\mathbb{T}^n\|_\mathbb{A}^{1/n}\leq \lim_{n\to \infty}\left\|\left(\widehat{\mathbb{T}}^\mathbb{A} \right)^n\right\|^{1/n}=r(\widehat{\mathbb{T}}^{\mathbb{A}}).$$
Therefore, we get \eqref{kk2020} as desired.
\end{proof}

Our second result in this section reads as follows.
\begin{theorem}\label{mainth1}
Let $T_1,T_2,S_1,S_2\in \mathcal{B}_{A^{1/2}}(\mathcal{H})$. Then,
\begin{align}\label{eq1}
 &r_A\left(T_1S_1+T_2S_2\right)\\
& \leq \frac{1}{2}\left[ \left\Vert S_1T_1\right\Vert_A +\left\Vert S_2T_2\right\Vert_A +%
\sqrt{\left( \left\Vert S_1T_1\right\Vert_A -\left\Vert S_2T_2\right\Vert_A \right)
^{2}+4\left\Vert S_1T_2\right\Vert_A \left\Vert S_2T_1\right\Vert_A}\right].\nonumber
\end{align}%
\end{theorem}
\begin{proof}
Notice first that, in general, for $T\in \mathcal{B}_{A^{1/2}}(\mathcal{H})$ and $\mathbb{A}=\begin{pmatrix}
A &0\\
0 &A
\end{pmatrix}$, we have
\begin{equation}\label{semmmmmmmm000}
\left\| \begin{pmatrix}
T &0\\
0 &0
\end{pmatrix}\right\|_{\mathbb{A}}=\|T\|_A.
\end{equation}
Indeed, it is not difficult to observe that
$$\left\| \begin{pmatrix}
T &0\\
0 &0
\end{pmatrix}\begin{pmatrix} x \\y\end{pmatrix}\right\|_{\mathbb{A}}=\|Tx\|_A,$$
for all $(x,y)\in \mathcal{H}\oplus \mathcal{H}$. So, we get \eqref{semmmmmmmm000} by taking the supremum over all $(x,y)\in \mathcal{H}\oplus \mathcal{H}$ with $\|x\|_A^2+\|y\|_A^2=1$ and using \eqref{aseminorm}.

Now, let $\mathbb{A}=\begin{pmatrix}
A &0\\
0 &A
\end{pmatrix}$. By using \eqref{semmmmmmmm000} together with \eqref{newrad} we see that
\begin{align*}
r_A\left(T_1S_1+T_2S_2\right)
& =r_{\mathbb{A}}\left[ \begin{pmatrix}
T_1S_1+T_2S_2 &0\\
0 &0
\end{pmatrix}\right] \\
 &=r_{\mathbb{A}}\left[ \begin{pmatrix}
T_1 &T_2\\
0 &0
\end{pmatrix} \begin{pmatrix}
S_1&0\\
S_2 &0
\end{pmatrix}\right]\\
 &=r_{\mathbb{A}}\left[\begin{pmatrix}
S_1&0\\
S_2 &0
\end{pmatrix} \begin{pmatrix}
T_1 &T_2\\
0 &0
\end{pmatrix} \right] \quad (\text{ by }\; \eqref{commut})\\
& =r_{\mathbb{A}}\left[ \begin{pmatrix}
S_1T_1 &S_1T_2\\
S_2T_1 &S_2T_2
\end{pmatrix}\right] \\
& \leq r\left[ \begin{pmatrix}
\|S_1T_1\|_A &\|S_1T_2\|_A\\
\|S_2T_1\|_A &\|S_2T_2\|_A
\end{pmatrix}\right] \quad (\text{ by Theorem }\ref{main222222})\\
&=\frac{1}{2}\Big( \left\Vert S_1T_1\right\Vert_A +\left\Vert S_2T_2\right\Vert_A \\
&\quad\quad\quad+\sqrt{\left( \left\Vert S_1T_1\right\Vert_A -\left\Vert S_2T_2\right\Vert_A \right)
^{2}+4\left\Vert S_1T_2\right\Vert_A \left\Vert S_2T_1\right\Vert_A }\Big)
\end{align*}
\end{proof}

\begin{corollary}\label{commmmm1}
Let $T,S\in \mathcal{B}_{A^{1/2}}(\mathcal{H})$. Then, we have
\begin{align*}
& r_A\left(TS\pm ST\right)\\
& \leq \frac{1}{2}\left( \left\Vert TS\right\Vert_A +\left\Vert ST\right\Vert_A +%
\sqrt{\left( \left\Vert TS\right\Vert_A -\left\Vert ST\right\Vert_A \right)
^{2}+4\left\Vert T^{2}\right\Vert_A \left\Vert S^{2}\right\Vert_A }\right)
\notag
\end{align*}%
\end{corollary}
\begin{proof}
By letting $T_1=S_2=T$, $S_1=S$ and $T_2=\pm S$ in Theorem \ref{mainth1} we get the desired result.
\end{proof}
Notice that Corollary \ref{commmmm1} provides an upper bound for the $A$-spectral radius of the commutator $TS-ST$.
\begin{corollary}\label{corfinal}
Let $U\in \mathcal{B}_{A}(\mathcal{H})$ be an $A$-unitary operator and $T\in \mathcal{B}_{A^{1/2}}(\mathcal{H})$. Then, we have
\begin{align*}
 r_A\left(TU\pm UT\right)\leq  \left\Vert T\right\Vert_A +\left\Vert T^2\right\Vert_A^{1/2}.
\end{align*}%
\end{corollary}
\begin{proof}
Notice first that since $U$ is an $A$-unitary operator, then
\begin{equation}
\|Ux\|_A=\|U^{\sharp_A}x\|_A=\|x\|_A.
\end{equation}
This implies, by using \eqref{aseminorm}, that
\begin{equation}\label{makkkk}
\|UT\|_A=\|U^{\sharp_A}T\|_A=\|T\|_A,\;\;\forall\,T\in \mathcal{B}_{A^{1/2}}(\mathcal{H}).
\end{equation}
Now, we will prove that $\|TU\|_A=\|T\|_A$. Clearly, we have
$$\{\|TUx\|_A\,;\;x\in \mathcal{H},\,\|x\|_A \}\subseteq \{\|Ty\|_A\,;\;y\in \mathcal{H},\,\|y\|_A \}.$$
So, by \eqref{aseminorm} we get
$$\|TU\|_A\leq\|T\|_A.$$
On the other hand, let
$$\lambda\in \{|\langle T^{\sharp_A}x\mid y\rangle_A|\,;\;x,y\in \mathcal{H},\,\|x\|_A=\|y\|_A=1 \},$$
 then there exist $x,y\in \mathcal{H}$ such that $\|x\|_A=\|y\|_A=1$ and $\lambda=|\langle T^{\sharp_A} x\mid y\rangle_A|$. Let $x=P_Ax+z_1$ and $y=P_Ay+z_2$ with $z_1,z_2\in \mathcal{N}(A)$. Since $T\in \mathcal{B}_{A^{1/2}}(\mathcal{H})$, then $\mathcal{N}(A)$ is an invariant subspace for each $T$. Hence, we obtain
\begin{align*}
\lambda
&=|\langle x\mid Ty\rangle_A|\\
&=|\langle P_Ax+z_1\mid AT(P_Ay+z_2)\rangle|\\
&=|\langle P_Ax\mid AT(P_Ay+z_2)\rangle|\\
&=|\langle P_Ax\mid AT(P_Ay)\rangle| \\
&=|\langle P_Ax\mid TP_Ay\rangle_A|=|\langle T^{\sharp_A} P_Ax\mid P_Ay\rangle_A|.
\end{align*}
Moreover, since $P_A=({U^{\sharp_A}})^{\sharp_A} U^{\sharp_A}$, it follows that
\begin{align*}
\lambda
&=|\langle T^{\sharp_A} ({U^{\sharp_A}})^{\sharp_A} U^{\sharp_A} x\mid ({U^{\sharp_A}})^{\sharp_A} U^{\sharp_A} y\rangle_A|\\
&=|\langle U^{\sharp_A} T^{\sharp_A} ({U^{\sharp_A}})^{\sharp_A} U^{\sharp_A} x\mid U^{\sharp_A} y\rangle_A|\\
&=|\langle [U^{\sharp_A} T U]^{\sharp_A} U^{\sharp_A} x\mid U^{\sharp_A} y\rangle_A|.
\end{align*}
Hence,
$$\lambda\in \left\{|\langle [U^{\sharp_A} T U]^{\sharp_A} z\mid t\rangle_A|\,;\;z,t\in \mathcal{H},\,\|z\|_A=\|t\|_A=1 \right\}.$$
This yields, by \eqref{semi2}, that $\|T^{\sharp_A}\|_A\leq \|[U^{\sharp_A} T U]^{\sharp_A}\|_A$ which in turn implies that
$$\|T\|_A\leq \|U^{\sharp_A} T U\|_A.$$
So, by \eqref{makkkk} we get $\|T\|_A\leq \|T U\|_A.$ Consequently, we have
\begin{equation}\label{makkkk2}
\|T\|_A=\|T U\|_A.
\end{equation}
Therefore, by letting $S=U$ in Corollary \ref{commmmm1} and then using \eqref{makkkk} together with \eqref{makkkk2} we see that
\begin{align*}
 r_A\left(TU\pm UT\right)\leq  \left\Vert T\right\Vert_A +\left\Vert T^2\right\Vert_A^{1/2}\left\Vert U^2\right\Vert_A^{1/2}.
\end{align*}%
This leads to the desired inequality since $\|U\|_A=1$.
\end{proof}

In order to prove our next result, we need the following lemma.
\begin{lemma}\label{tawtaw}
Let $T,S\in \mathcal{B}_{A^{1/2}}(\mathcal{H})$ and $\mathbb{A}=\begin{pmatrix}A&0\\0&A\end{pmatrix}$. Then, the following assertions hold:
\begin{itemize}
  \item [(1)] $r_{\mathbb{A}}\left[ \begin{pmatrix}
T &0\\
0 &S
\end{pmatrix}\right]=\max\{r_A(T),r_A(S)\}$.
  \item [(2)] $r_{\mathbb{A}}\left[ \begin{pmatrix}
0 &T\\
S &0
\end{pmatrix}\right]=\sqrt{r_A(TS)}$.
\end{itemize}
\end{lemma}
\begin{proof}
\noindent (1)\;It can observed that
$$\left[ \begin{pmatrix}
T &0\\
0 &S
\end{pmatrix}\right]^n=\begin{pmatrix}
T^n &0\\
0 &S^n
\end{pmatrix},\;\;\forall\,n\in \mathbb{N}^*.$$
Moreover, for every $(x,y)\in \mathcal{H}\oplus \mathcal{H}$ we have
\begin{align*}
\left\|\begin{pmatrix}
T^n &0\\
0 &S^n
\end{pmatrix}\begin{pmatrix}
x \\
y
\end{pmatrix}\right\|_{\mathbb{A}}^2
& =\left \|\begin{pmatrix}
T^nx \\
T^ny
\end{pmatrix}\right\|_{\mathbb{A}}^2\\
 &=\|T^nx\|_A^2+\|S^ny\|_A^2\\
  &\leq \max\left\{\|T^n\|_A^2,\|S^n\|_A^2 \right\} (\|x\|_A^2+\|y\|_A^2).
\end{align*}
This implies, by using \eqref{aseminorm}, that
$$\left\|\begin{pmatrix}
T^n &0\\
0 &S^n
\end{pmatrix}\right\|_{\mathbb{A}} \leq \max\left\{\|T^n\|_A,\|S^n\|_A \right\}.$$
Let $(x,0) \in  \mathcal{H}\oplus \mathcal{H}$ be such that $\|x\|_A =1.$ Then
$$\left\|\begin{pmatrix}
T^n &0\\
0 &S^n
\end{pmatrix}\right\|_{\mathbb{A}}\geq\left\|\begin{pmatrix}
T^n &0\\
0 &S^n
\end{pmatrix}\begin{pmatrix}
x \\
0
\end{pmatrix}\right\|_{\mathbb{A}}=\|T^nx\|_A.$$
So, by taking the supremum over all $x\in \mathcal{H}$ with $\|x\|_A=1$, we get
$$\left\|\begin{pmatrix}
T^n &0\\
0 &S^n
\end{pmatrix}\right\|_{\mathbb{A}}\geq \|T^n\|_A.$$
Similarly, if we take $(0,y) \in \mathcal{H} \oplus \mathcal{H}$ with $\|y\|_A=1$, we get
$$\left\|\begin{pmatrix}
T^n &0\\
0 &S^n
\end{pmatrix}\right\|_{\mathbb{A}}\geq \|S^n\|_A.$$
Hence,
\begin{equation*}
\left\|\begin{pmatrix}
T^n &0\\
0 &S^n
\end{pmatrix}\right\|_{\mathbb{A}} =\max\left\{\|T^n\|_A,\|S^n\|_A \right\},
\end{equation*}
for all $n\in \mathbb{N}^*$. Hence, the proof of the first assertion is finished by using \eqref{newrad}.

\par \vskip 0.1 cm \noindent (2)\;By using the first assertion and Proposition \ref{22} we see that
$$r_{\mathbb{A}}^2\left[ \begin{pmatrix}
0 &T\\
S &0
\end{pmatrix}\right]=r_{\mathbb{A}}\left[ \begin{pmatrix}
0 &T\\
S &0
\end{pmatrix}^2\right]=r_{\mathbb{A}}\left[ \begin{pmatrix}
TS &0\\
0 &ST
\end{pmatrix}\right]=\max\{r_A(TS),r_A(ST)\}.$$
However, by \eqref{commut} we have $r_A(TS)=r_A(ST)$. Therefore, the proof is complete.
\end{proof}
Now, we state the following theorem.
\begin{theorem}
Let $\mathbb{T}=\begin{pmatrix}
P &Q\\
R &S
\end{pmatrix}$ be such that $P, Q, R, S\in \mathcal{B}_{A^{1/2}}(\mathcal{H})$ and $\mathbb{A}=\begin{pmatrix}A&0\\0&A\end{pmatrix}$. Then,
\begin{equation}
\max\Big\{\sqrt{r_A(QR)} ,\max\left\{r_A(P),r_A(S)\right\}\Big\}\leq \frac{1}{2}\left(\|\mathbb{T}\|_{\mathbb{A}}+\|\mathbb{T}^2\|_{\mathbb{A}}^{1/2}\right).
\end{equation}
\end{theorem}
\begin{proof}
Let $\mathbb{U}=\left(\begin{array}{cc}
I&O \\
O&-I
\end{array}\right).$ In view of \cite[Lemma 3.1.]{bhunfekipaul}, we have $\mathbb{U}\in \mathcal{B}_{\mathbb{A}}(\mathcal{H}\oplus \mathcal{H})$ and $\mathbb{U}^{\sharp_{\mathbb{A}}}=\left(\begin{array}{cc}
I^{\sharp_A}&O \\
O&(-I)^{\sharp_A}
\end{array}\right).$ So, ones get
$$\mathbb{U}^{\sharp_{\mathbb{A}}}\mathbb{U}=\left(\begin{array}{cc}
P_A&O \\
O&-P_A
\end{array}\right)\left(\begin{array}{cc}
I&O \\
O&-I
\end{array}\right)=\left(\begin{array}{cc}
P_A&O \\
O&P_A
\end{array}\right)=P_{\mathbb{A}},$$
Similarly, we show that $(\mathbb{U}^{\sharp_\mathbb{A}})^{\sharp_\mathbb{A}} \mathbb{U}^{\sharp_\mathbb{A}}=P_\mathbb{A}.$ Hence, $\mathbb{U}$ is an $\mathbb{A}$-unitary operator. Moreover, it can verified that
\begin{equation}\label{jdidtaw}
\mathbb{T}\mathbb{U}+\mathbb{U}\mathbb{T}=2\begin{pmatrix}
P &0\\
0 &-S
\end{pmatrix}\;\text{ and }\;\mathbb{T}\mathbb{U}-\mathbb{U}\mathbb{T}=2\begin{pmatrix}
0 &-Q\\
R &0
\end{pmatrix}.
\end{equation}
So, by using Lemma \ref{tawtaw} together with \eqref{jdidtaw} we get
\begin{align*}
2\max\left\{r_A(P),r_A(S)\right\}
& =r_{\mathbb{A}} \left[\begin{pmatrix}
2P &0\\
0 &-2S
\end{pmatrix}\right]\\
 &=r_{\mathbb{A}}(\mathbb{T}\mathbb{U}+\mathbb{U}\mathbb{T})\\
  &\leq\|\mathbb{T}\|_{\mathbb{A}}+\|\mathbb{T}^2\|_{\mathbb{A}}^{1/2},\;(\text{by Corollary} \ref{corfinal}).
\end{align*}
Thus, we get
\begin{equation}\label{f1}
\max\left\{r_A(P),r_A(S)\right\}\leq \frac{1}{2}\left(\|\mathbb{T}\|_{\mathbb{A}}+\|\mathbb{T}^2\|_{\mathbb{A}}^{1/2}\right).
\end{equation}
On the other hand, again by using Lemma \ref{tawtaw} together with \eqref{jdidtaw} we get
\begin{align*}
2\sqrt{r_A(QR)}
& =r_{\mathbb{A}} \left[\begin{pmatrix}
0 &-2Q\\
2R &0
\end{pmatrix}\right]\\
 &=r_{\mathbb{A}}(\mathbb{T}\mathbb{U}-\mathbb{U}\mathbb{T})\\
 &\leq\|\mathbb{T}\|_{\mathbb{A}}+\|\mathbb{T}^2\|_{\mathbb{A}}^{1/2},\;(\text{by Corollary} \ref{corfinal}).
\end{align*}
Thus, we get
\begin{equation}\label{f2}
\sqrt{r_A(QR)}\leq \frac{1}{2}\left(\|\mathbb{T}\|_{\mathbb{A}}+\|\mathbb{T}^2\|_{\mathbb{A}}^{1/2}\right).
\end{equation}
Combining \eqref{f1} together with \eqref{f2} yields to the desired result.
\end{proof}

In order to establish a new $A$-spectral radius inequality, we need to recall from \cite{BP} the following lemma.
\begin{lemma}\label{cor-1}
Let $\mathbb{T}=\begin{pmatrix}
P &Q\\
R &S
\end{pmatrix}$ be such that $P, Q, R, S\in \mathcal{B}_{A^{1/2}}(\mathcal{H})$ and $\mathbb{A}=\begin{pmatrix}A&0\\0&A\end{pmatrix}$. Then
\[\omega_\mathbb{A}(\mathbb{T})\leq \frac{1}{2}\left[ \omega_A(P)+\omega_A(S)+\sqrt{\big(\omega_A(P)-\omega_A(S)\big)^2+\big(\|Q\|_A+\|R\|_A\big)^2}\right].\]
\end{lemma}

\begin{theorem}\label{main2jd}
Let $T_1,T_2,S_1,S_2\in \mathcal{B}_{A^{1/2}}(\mathcal{H})$. Then,
\begin{align*}\label{eq2}
&r_A\left(T_1S_1+T_2S_2\right)\nonumber\\
& \leq\frac{1}{2}\left[\omega_A(S_1T_1) +\omega_A(S_2T_2)\big)+
\sqrt{\left( \omega_A(S_1T_1) -\omega_A(S_2T_2)\right)
^{2}+4\left\Vert S_1T_2\right\Vert_A \left\Vert S_2T_1\right\Vert_A }\right].
\end{align*}%
Moreover, this inequality refines \eqref{eq1}.
\end{theorem}
\begin{proof}
Let $\mathbb{A}=\begin{pmatrix}A&0\\0&A\end{pmatrix}$. By proceeding as in the proof of Theorem \ref{mainth1} and using \eqref{dom} we get
\begin{align*}
r_A\left(T_1S_1+T_2S_2\right)
&=r_{\mathbb{A}}\left[ \begin{pmatrix}
S_1T_1 &S_1T_2\\
S_2T_1 &S_2T_2
\end{pmatrix}\right]\\
&\leq \omega_{\mathbb{A}}\left[ \begin{pmatrix}
S_1T_1 &S_1T_2\\
S_2T_1 &S_2T_2
\end{pmatrix}\right].
\end{align*}
So, by applying Lemma \ref{cor-1} we obtain
\begin{align*}
r_A\left(T_1S_1+T_2S_2\right)
&\leq\frac{1}{2}\big( \omega_A(S_1T_1) +\omega_A(S_2T_2)\big)\nonumber\\
&+\frac{1}{2}\Big(\sqrt{\left( \omega_A(S_1T_1) -\omega_A(S_2T_2)\right)
^{2}+\big(\|S_1T_2\|_A +\|S_2T_1\|_A\big)^2 }\Big).
\end{align*}

On the other hand, it can be observed that for every positive real numbers $\alpha$ and $\beta$ we have
\begin{equation}\label{infkitta}
\inf\left\{\,\alpha t+\tfrac{\beta}{t}\,:\,t\in \mathbb{R}, t>0\,\right\}=2\sqrt{\alpha\beta}.
\end{equation}
If $\left\Vert S_1T_2\right\Vert_A=0$ or $\left\Vert S_2T_1\right\Vert_A=0$, then the required inequality holds trivially. Assume that $\left\Vert S_1T_2\right\Vert_A\neq0$ and $\left\Vert S_2T_1\right\Vert_A\neq0$. By replacing $T_1$ and $S_1$ by $\varepsilon T_1$ and $\frac{1}{\varepsilon}S_1$ in the last inequality respectively, and then taking the infimum over $\varepsilon> 0$ and using \eqref{infkitta} we get the desired inequality.

Now, in order to see that the obtained inequality refines \eqref{eq1}, we let
$$a_1=\frac{1}{2}\left[ \left\Vert S_1T_1\right\Vert +\left\Vert S_2T_2\right\Vert +%
\sqrt{\left( \left\Vert S_1T_1\right\Vert -\left\Vert S_2T_2\right\Vert \right)
^{2}+4\left\Vert S_1T_2\right\Vert \left\Vert S_2T_1\right\Vert }\right]\;\text{ and }$$
$$a_2=\frac{1}{2}\left[\omega_A(S_1T_1) +\omega_A(S_2T_2)\big)+
\sqrt{\left( \omega_A(S_1T_1) -\omega_A(S_2T_2)\right)
^{2}+4\left\Vert S_1T_2\right\Vert_A \left\Vert S_2T_1\right\Vert_A }\right].$$
It is not difficult to observe that
\begin{align*}
a_1
&=r\left[ \begin{pmatrix}
\|S_1T_1\|_A &\sqrt{\left\Vert S_1T_2\right\Vert_A \left\Vert S_2T_1\right\Vert_A}\\
\sqrt{\left\Vert S_1T_2\right\Vert_A \left\Vert S_2T_1\right\Vert_A} &\|S_2T_2\|_A
\end{pmatrix}\right]\\
&=\left\| \begin{pmatrix}
\|S_1T_1\|_A &\sqrt{\left\Vert S_1T_2\right\Vert_A \left\Vert S_2T_1\right\Vert_A}\\
\sqrt{\left\Vert S_1T_2\right\Vert_A \left\Vert S_2T_1\right\Vert_A} &\|S_2T_2\|_A
\end{pmatrix}\right\|,
\end{align*}
and
\begin{align*}
a_2
&=r\left[ \begin{pmatrix}
\omega_A(S_1T_1) &\sqrt{\left\Vert S_1T_2\right\Vert_A \left\Vert S_2T_1\right\Vert_A}\\
\sqrt{\left\Vert S_1T_2\right\Vert_A \left\Vert S_2T_1\right\Vert_A} &\omega_A(S_2T_2)
\end{pmatrix}\right]\\
&=\left\| \begin{pmatrix}
\omega_A(S_1T_1) &\sqrt{\left\Vert S_1T_2\right\Vert_A \left\Vert S_2T_1\right\Vert_A}\\
\sqrt{\left\Vert S_1T_2\right\Vert_A \left\Vert S_2T_1\right\Vert_A} &\omega_A(S_2T_2)
\end{pmatrix}\right\|
\end{align*}
Since $\omega_A(X)\leq \|X\|_A$ for all $X\in \mathcal{B}_{A^{1/2}}(\mathcal{H})$, then it follows from the norm monotonicity of matrices with nonnegative entries that $a_2\leq a_1$.
\end{proof}
\begin{corollary}
Let $T,S\in \mathcal{B}_{A^{1/2}}(\mathcal{H})$. Then, we have
\begin{align*}
 r_A\left(T+S\right)\leq \frac{1}{2}\left( \omega_A(T)+\omega_A(S)+%
\sqrt{\left[ \omega_A(T) -\omega_A(S) \right]^{2}+4\mu(T,S) }\right),
\notag
\end{align*}%
where $\mu(T,S)=\min\{\|TS\|_A,\|ST\|_A\}$.
\end{corollary}
\begin{proof}
By letting $T_1=T$, $S_2=S$ and $T_2=S_1=I$ in Theorem \ref{main2jd} we get
\begin{equation}\label{l1}
r_A\left(T+S\right)\leq \frac{1}{2}\left( \omega_A(T)+\omega_A(S)+%
\sqrt{\left[ \omega_A(T) -\omega_A(S) \right]^{2}+4\|ST\|_A }\right).
\end{equation}
This implies, by symmetry, that
\begin{equation}\label{l2}
r_A\left(T+S\right)\leq \frac{1}{2}\left( \omega_A(T)+\omega_A(S)+%
\sqrt{\left[ \omega_A(T) -\omega_A(S) \right]^{2}+4\|TS\|_A }\right).
\end{equation}
So, we get the desired result by combining \eqref{l1} together with \eqref{l2}.
\end{proof}
\begin{corollary}
Let $T,S\in \mathcal{B}_{A^{1/2}}(\mathcal{H})$. Then, we have
\begin{align*}
 r_A\left(TS\right)\leq \frac{1}{4}\left( \omega_A(TS)+\omega_A(ST)+%
\sqrt{\left[ \omega_A(TS) -\omega_A(ST) \right]^{2}+4\nu(T,S) }\right),
\notag
\end{align*}%
where $\nu(T,S)=\min\{\|T\|_A\|STS\|_A,\|S\|_A\|TST\|_A\}$.
\end{corollary}
\begin{proof}
By letting $T_1=\frac{T}{2}$, $S_1=S$, $T_2=\frac{TS}{2}$ and $S_2=I$ in Theorem \ref{main2jd} we obtain
\begin{equation*}
r_A\left(TS\right)\leq \frac{1}{4}\left( \omega_A(TS)+\omega_A(ST)+%
\sqrt{\left[ \omega_A(TS) -\omega_A(ST) \right]^{2}+4\|T\|_A\|STS\|_A }\right).
\end{equation*}
This in turn implies, by symmetry, that
\begin{equation*}
r_A\left(TS\right)\leq \frac{1}{4}\left( \omega_A(TS)+\omega_A(ST)+%
\sqrt{\left[ \omega_A(TS) -\omega_A(ST) \right]^{2}+4\|S\|_A\|TST\|_A }\right).
\end{equation*}
Therefore, the desired inequality follows immediately from the above two inequalities.
\end{proof}

\begin{corollary}
Let $T,S\in \mathcal{B}_{A^{1/2}}(\mathcal{H})$. Then, we have
\begin{align}\label{e.0.7}
& r_A\left(TS\pm ST\right)\\
& \leq \frac{1}{2}\left( \omega_A(TS)+\omega_A(ST)+%
\sqrt{\left[ \omega_A(TS) -\omega_A(ST) \right]^{2}+4\|T^{2}\|_A \|S^{2}\|_A }\right)
\notag
\end{align}%
Moreover, if $TS=ST$, then
\begin{equation}\label{e.0.8}
r_A\left( TS\right) \leq \frac{1}{2}\left( \omega_A(TS)+\left\Vert T^{2}\right\Vert_A ^{1/2}\left\Vert S^{2}\right\Vert_A ^{1/2}\right).
\end{equation}%
\end{corollary}
\begin{proof}
The inequality \eqref{e.0.7} follows from Theorem \ref{main2jd} by letting $T_1=S_2=T$, $S_1=S$ and $T_2=\pm S$. Moreover, if $TS=ST$, then \eqref{e.0.8} holds immediately by using \eqref{e.0.7}.
\end{proof}

The following proposition is also an immediate consequence of Theorem \ref{main2jd}.
\begin{proposition}
Let $T,S\in \mathcal{B}_{A^{1/2}}(\mathcal{H})$. Then, we have
\begin{align}\label{e.0.9}
r_A\left( TS\pm ST\right) &\leq \omega_A(TS)+\min \left\{
\left\Vert T\right\Vert_A ^{1/2}\left\Vert TS^{2}\right\Vert_A ^{1/2},\left\Vert
T^{2}S\right\Vert_A^{1/2}\left\Vert S\right\Vert_A ^{1/2}\right\}.
\end{align}%
and
\begin{align}\label{e.0.10}
r_A\left( TS\pm ST\right) &\leq \omega_A(ST)+\min \left\{
\left\Vert S\right\Vert_A ^{1/2}\left\Vert ST^{2}\right\Vert_A ^{1/2},\left\Vert
S^{2}T\right\Vert_A^{1/2}\left\Vert T\right\Vert_A ^{1/2}\right\}.
\end{align}%
\end{proposition}
\begin{proof}
By letting $T_1=I$, $T_2=S$, $S_1=TS$ and $S_2=\pm T$ in Theorem \ref{main2jd} we get
\begin{align}\label{ej1}
r_A\left( TS\pm ST\right) &\leq \omega_A(TS)+
\left\Vert T\right\Vert_A ^{1/2}\left\Vert TS^{2}\right\Vert_A ^{1/2}.
\end{align}
On the other hand, similarly by letting $T_1=TS$, $T_2=S$, $S_1=I$ and $S_2=\pm T$ in Theorem \ref{main2jd} we obtain
\begin{align}\label{ej2}
r_A\left( TS\pm ST\right) &\leq \omega_A(TS)+\left\Vert
T^{2}S\right\Vert_A^{1/2}\left\Vert S\right\Vert_A ^{1/2}.
\end{align}
So, the inequality \eqref{e.0.9} follows immediately by using \eqref{ej1} and \eqref{ej2}. In addition, the inequality \eqref{e.0.10} follows from \eqref{e.0.9} by symmetry.
\end{proof}
\begin{corollary}
Let $T,S\in \mathcal{B}_{A^{1/2}}(\mathcal{H})$ be such that $TS=ST$. Then,
\begin{align}\label{23jun}
r_A\left( TS\right) & \leq \frac{1}{2}\left[ \omega_A(TS)+
\left\Vert T\right\Vert_A ^{1/2}\left\Vert S\right\Vert_A ^{1/2}\left\Vert
TS\right\Vert_A ^{1/2} \right].
\end{align}
and
\begin{align}\label{23jun2}
r_A\left( TS\right) & \leq \frac{1}{2}\left[ \omega_A(TS)+
\min \left\{ \left\Vert T\right\Vert_A \left\Vert S^{2}\right\Vert_A
^{1/2},\left\Vert T^{2}\right\Vert_A ^{1/2}\left\Vert S\right\Vert_A \right\}\right].
\end{align}
\end{corollary}
\begin{proof}
Since $TS=ST$, then it follows from \eqref{e.0.9} that
\begin{align}\label{tawtaw}
r_A\left( TS\right) & \leq \frac{1}{2}\left[ \left\Vert TS\right\Vert_A +
\min\left\{ \left\Vert T\right\Vert_A ^{1/2}\left\Vert TS^{2}\right\Vert_A
^{1/2},\left\Vert T^{2}S\right\Vert_A ^{1/2}\left\Vert S\right\Vert_A
^{1/2}\right\} \right].
\end{align}
On the other hand, by using \eqref{semiiineq2} we see that
$$\left\Vert T\right\Vert_A ^{1/2}\left\Vert TS^{2}\right\Vert_A
^{1/2}\leq \left\Vert T\right\Vert_A ^{1/2}\left\Vert S\right\Vert_A ^{1/2}\left\Vert
TS\right\Vert_A ^{1/2},$$
and
$$\left\Vert T^{2}S\right\Vert_A ^{1/2}\left\Vert S\right\Vert_A
^{1/2}\leq \left\Vert T\right\Vert_A ^{1/2}\left\Vert S\right\Vert_A ^{1/2}\left\Vert
TS\right\Vert_A ^{1/2}.$$
So, we infer that
$$\min\left\{ \left\Vert T\right\Vert_A ^{1/2}\left\Vert TS^{2}\right\Vert_A
^{1/2},\left\Vert T^{2}S\right\Vert_A ^{1/2}\left\Vert S\right\Vert_A
^{1/2}\right\}\leq \left\Vert T\right\Vert_A ^{1/2}\left\Vert S\right\Vert_A ^{1/2}\left\Vert
TS\right\Vert_A ^{1/2}.$$
Hence, by taking into account \eqref{tawtaw}, we get \eqref{23jun} as required. Moreover, \eqref{23jun2} follows immediately by using \eqref{tawtaw} together with \eqref{semiiineq2}.
\end{proof}

Now, in order to prove our next result which also a consequence of Theorem \ref{main2jd} we need to recall from \cite{zamani1} the following lemma.
\begin{lemma}\label{T.2.30}
Let $T\in\mathbb{B}_{A}(\mathcal{H})$. Then
\begin{align*}
\omega_A(T) = \displaystyle{\sup_{\theta \in \mathbb{R}}}{\left\|\Re_A(e^{i\theta}T)\right\|}_A\,,\; \text{ where }\; \Re_A(e^{i\theta}T)=\frac{e^{i\theta}T + e^{-i\theta}T^{\sharp_A}}{2}.
\end{align*}
\end{lemma}
Our next result is stated as follows.
\begin{theorem}
Let $T,S\in\mathbb{B}_{A}(\mathcal{H})$. Then,
\begin{equation}
\omega_A(TS)\leq \frac{1}{2}\Big(\omega_A(ST)+\|T\|_A\|S\|_A\Big).
\end{equation}
\end{theorem}
\begin{proof}
Let $\theta \in \mathbb{R}$. It can be seen that $\Re_A(e^{i\theta}TS)$ is an $A$-self-adjoint operator (that is $A\Re_A(e^{i\theta}TS)$ is a self-adjoint operator). So, by \cite{feki01} we deduce that
$$\|\Re_A(e^{i\theta}TS)\|_A=r_A\Big(\Re_A(e^{i\theta}TS)\Big).$$
On the other hand, we have
$$r_A\Big(\Re_A(e^{i\theta}TS)\Big)=\frac{1}{2}r_A\Big(e^{i\theta}TS+e^{-i\theta}(TS)^{\sharp_A}\Big)=
\frac{1}{2}r_A\Big(e^{i\theta}TS+e^{-i\theta}S^{\sharp_A}T^{\sharp_A}\Big).$$
By letting $T_1=e^{i\theta}T$, $S_1=S$, $T_2=e^{-i\theta}S^{\sharp_A}$ and $S_2=T^{\sharp_A}$ in Theorem \ref{main2jd}, we get
$$r_A\Big(e^{i\theta}TS+e^{-i\theta}S^{\sharp_A}T^{\sharp_A}\Big)\leq \omega_A(ST)+\|T\|_A\|S\|_A.$$
Hence,
\begin{align*}
\omega_A(TS)
& =\displaystyle{\sup_{\theta \in \mathbb{R}}}{\left\|\Re_A(e^{i\theta}TS)\right\|}_A \\
 &\leq \frac{1}{2}\Big(\omega_A(ST)+\|T\|_A\|S\|_A\Big).
\end{align*}
\end{proof}

Now, we turn your attention to establish an estimate for the $A$-spectral radius of the sum of product of a $d$-pairs of operators. In order to achieve our goal, we need some prerequisites.

The semi-inner product ${\langle \cdot\mid \cdot\rangle}_{A}$ induces on the quotient $\mathcal{H}/\mathcal{N}(A)$ an inner product which is not complete unless $\mathcal{R}(A)$ is closed. However, it was shown in \cite{branrov} that the completion of $\mathcal{H}/\mathcal{N}(A)$ is isometrically isomorphic to the Hilbert space $\mathbf{R}(A^{1/2}):=\big(\mathcal{R}(A^{1/2}), \langle\cdot,\cdot\rangle_{\mathbf{R}(A^{1/2})}\big)$ such that
\begin{align*}
\langle A^{1/2}x,A^{1/2}y\rangle_{\mathbf{R}(A^{1/2})}:=\langle P_Ax\mid P_Ay\rangle,\;\forall\, x,y \in \mathcal{H}.
\end{align*}
For more information related to the Hilbert space $\mathbf{R}(A^{1/2})$, we refer the reader to \cite{acg3,majsecesuci} and the references therein. The following proposition is taken from \cite{acg3}.
\begin{proposition}\label{prop_arias}
Let $T\in \mathcal{B}(\mathcal{H})$. Then $T\in \mathcal{B}_{A^{1/2}}(\mathcal{H})$ if and only if there exists a unique $\widetilde{T}\in \mathcal{B}(\mathbf{R}(A^{1/2}))$ such that $Z_AT =\widetilde{T}Z_A$, where
$$Z_A\colon\mathcal{H}\to \mathbf{R}(A^{1/2}),\;x\mapsto Z_Ax:=Ax.$$
\end{proposition}
In addition, for $T\in \mathcal{B}_{A^{1/2}}(\mathcal{H})$, we have
\begin{equation}\label{acg2009}
\|T\|_A=\|\widetilde{T}\|_{\mathcal{B}(\mathbf{R}(A^{1/2}))},
\end{equation}
(see \cite[Proposition 3.9]{acg3}). Also, we need the following Lemma.

\begin{lemma}\label{kais}(\cite{feki01})
If $T\in \mathcal{B}_{A^{1/2}}(\mathcal{H})$, then $\omega_A(T)=\omega(\widetilde{T})$.
\end{lemma}

Now, we are in a position to prove the following upper bound for the $\mathbb{A}$-numerical radius of $d\times d$ operator matrices which allows us to establish an estimate for the $A$-spectral radius of the sum of products of a $d$-pairs of operators.

\begin{theorem}\label{theorem:upperbound3}
 Let $\mathbb{T}= (T_{ij})_{d \times d}$ be such that $T_{ij}\in \mathcal{B}_{A^{1/2}}(\mathcal{H})$ for all $i,j$. Then,
\[\omega_{\mathbb{A}}(\mathbb{T})\leq \max_{1\leq i \leq d}\left\{ \omega_A(T_{ii})+\frac{1}{2}\sum^d_{j=1,j\neq i}(\|T_{ij}\|_A+\|T_{ji}\|_A)\right\}.\]
\end{theorem}
\begin{proof}
Notice first that since $T_{ij}\in \mathcal{B}_{A^{1/2}}(\mathcal{H})$ for all $i,j$, then by Lemma \ref{lemmajdid} we have $\mathbb{T}\in\mathcal{B}_{\mathbb{A}^{1/2}}(\mathbb{H})$. So, by Proposition \ref{prop_arias} there exists a unique $\widetilde{\mathbb{T}}\in \mathcal{B}(\mathbf{R}(\mathbb{A}^{1/2}))$ such that $Z_{\mathbb{A}}\mathbb{T} =\widetilde{\mathbb{T}}Z_{\mathbb{A}}$. On the other hand, since $T_{ij}\in \mathcal{B}_{A^{1/2}}(\mathcal{H})$ for all $i,j$, then by Proposition \ref{prop_arias} there exists a unique $\widetilde{T_{ij}}\in \mathcal{B}(\mathbf{R}(A^{1/2}))$ such that $Z_{A}T_{ij} =\widetilde{T_{ij}}Z_{A}$. Let $\mathbb{S}=(\widetilde{T_{ij}})_{d \times d}$. It is not difficult to verify that $Z_{\mathbb{A}}\mathbb{T} =\mathbb{S}Z_{\mathbb{A}}$. So, we infer that
\begin{equation}\label{jdid}
\widetilde{\mathbb{T}}=\mathbb{S}=(\widetilde{T_{ij}})_{d \times d}.
\end{equation}
Now, since $\mathbb{S}=(\widetilde{T_{ij}})_{d \times d}$ is a $d \times d$ operator matrix with $\widetilde{T_{ij}}\in \mathcal{B}(\mathbf{R}(A^{1/2}))$ for all $i,j$, then by using \cite[Theorem 2]{OK2} together with \cite[Remark 1]{OK2} we observe that
\[\omega(\mathbb{S})\leq \max_{1\leq i \leq d}\left\{ \omega(\widetilde{T_{ii}})+\frac{1}{2}\sum^d_{j=1,j\neq i}(\|\widetilde{T_{ij}}\|_{\mathcal{B}(\mathbf{R}(A^{1/2}))}+\|\widetilde{T_{ji}}\|_{\mathcal{B}(\mathbf{R}(A^{1/2}))})\right\}.\]
This implies, by applying Lemma \ref{kais} in combination with \eqref{jdid} and \eqref{acg2009}, that
\begin{align*}
\omega_{\mathbb{A}}(\mathbb{T})
&=\omega(\mathbb{S})\\
&\leq \max_{1\leq i \leq d}\left\{ \omega_A(T_{ii})+\frac{1}{2}\sum^d_{j=1,j\neq i}(\|T_{ij}\|_A+\|T_{ji}\|_A)\right\},
\end{align*}
as required. Hence, the proof is complete.
\end{proof}

Now we can state and prove the following result which extends \cite[Theorem 2.10.]{BBP}.
\begin{theorem}\label{theorem:spectral}
Let $(T_1,\cdots,T_d)\in \mathcal{B}_{A^{1/2}}(\mathcal{H})^d$ and $(S_1,\cdots,S_d)\in \mathcal{B}_{A^{1/2}}(\mathcal{H})^d$. Then,
\[r_A \left(\sum^d_{k=1}T_kS_k\right)\leq \max_{1\leq i\leq d}\left\{ \omega_A(S_iT_i)+\frac{1}{2}\sum^d_{j=1,j\neq i}\left(\|S_iT_j\|_A+\|S_jT_i\|_A\right)\right\}.\]

\end{theorem}
\begin{proof}
Notice first that, it can be verified that
$$\left\|\left(\begin{array}{cccccc}
    X&0&.&.&.&0 \\
    0&0&.&.&.&0 \\
    .& & & & &  \\
    .& & & & & \\
    .& & & & & \\
    0&0&.&.&.&0
    \end{array}\right)\right\|_{\mathbb{A}}=\|X\|_A,$$
for all $X\in \mathcal{B}_{A^{1/2}}(\mathcal{H})$. So, we see that
\begin{align*}
r_A \left(\sum^d_{i=1}T_iS_i\right)
&= r_\mathbb{A} \left(\begin{array}{cccccc}
    \sum^d_{i=1}T_iS_i&0&.&.&.&0 \\
    0&0&.&.&.&0 \\
    .& & & & &  \\
    .& & & & & \\
    .& & & & & \\
    0&0&.&.&.&0
    \end{array}\right) \\
		&= r_\mathbb{A} \left[\left(\begin{array}{cccccc}
    T_1&T_2&.&.&.&T_d \\
    0&0&.&.&.&0 \\
    .& & & & &  \\
    .& & & & & \\
    .& & & & & \\
    0&0&.&.&.&0
    \end{array}\right)\left(\begin{array}{cccccc}
    S_1&0&.&.&.&0 \\
    S_2&0&.&.&.&0 \\
    .& & & & &  \\
    .& & & & & \\
    .& & & & & \\
    S_d&0&.&.&.&0
    \end{array}\right)\right]\\
		&= r_\mathbb{A} \left[\left(\begin{array}{cccccc}
    S_1&0&.&.&.&0 \\
    S_2&0&.&.&.&0 \\
    .& & & & &  \\
    .& & & & & \\
    .& & & & & \\
    S_d&0&.&.&.&0
    \end{array}\right)\left(\begin{array}{cccccc}
    T_1&T_2&.&.&.&T_d \\
    0&0&.&.&.&0 \\
    .& & & & &  \\
    .& & & & & \\
    .& & & & & \\
    0&0&.&.&.&0
    \end{array}\right)\right]\; (\text{ by }\; \eqref{commut}).
\end{align*}
Hence, by using \eqref{dom} and Theorem \ref{theorem:upperbound3} we get
 \begin{align*}
r_A \left(\sum^d_{k=1}T_kS_k\right)
		&= r_\mathbb{A} \left(\begin{array}{cccccc}
    S_1T_1&S_1T_2&.&.&.&S_1T_d \\
    S_2T_1&S_2T_2&.&.&.&S_2T_d\\
    .& & & & &  \\
    .& & & & & \\
    .& & & & & \\
    S_dT_1&S_dT_2&.&.&.&S_dT_d
    \end{array}\right)\\
		&\leq  \omega_\mathbb{A}\left(\begin{array}{cccccc}
    S_1T_1&S_1T_2&.&.&.&S_1T_d \\
    S_2T_1&S_2T_2&.&.&.&S_2T_d\\
    .& & & & &  \\
    .& & & & & \\
    .& & & & & \\
    S_dT_1&S_dT_2&.&.&.&S_dT_d
    \end{array}\right)\\
		&\leq \max_{1\leq i\leq d}\{ \omega_A(S_iT_i)+\frac{1}{2}\sum^d_{j=1,j\neq i}\left(\|S_iT_j\|_A+\|S_jT_i\|_A\right)\}.
\end{align*}
Hence, the proof is complete.
\end{proof}
\section{$A$-spectral radius for functions of operators}\label{s3}
In this section, we discuss the $A$-spectral radius of an operator which is defined by the help of power series $f\left(z\right) =\sum_{n=0}^{\infty
}c_{n}z^{n}$ such that $c_k$ are complex coefficients for all $k$. In order to achieve our goal in this section, we need some lemmas.
\begin{lemma}\label{l.2.2}
Let $T,S\in \mathcal{B}_{A^{1/2}}(\mathcal{H})$ be such that $TS=ST$. Then,
\begin{equation}
\left\vert r_A\left( T\right) -r_A\left( S\right) \right\vert \leq r_A\left(
T-S\right).  \label{e.2.2}
\end{equation}
\end{lemma}

\begin{proof}
Since $TS=ST$, then clearly $T-S$ and $S$ commute. So by Proposition \ref{22} we have%
\begin{equation*}
r_A\left(T\right) =r_A\left(T-S+S\right) \leq r_A\left(T-S\right) +r_A\left(
S\right).
\end{equation*}%
This immediately gives
\begin{equation}
r_A\left(T\right) -r_A\left( S\right) \leq r_A\left(T-S\right).  \label{e.2.3}
\end{equation}%
On the other hand, since $S-T$ and $T$ commute, then again by applying Proposition \ref{22}, we get
\begin{equation*}
r_A\left( S\right) \leq r_A\left( S-T\right) +r_A\left(T\right),
\end{equation*}%
which in turn implies that
\begin{equation*}
r_A\left( S\right) -r_A\left(T\right) \leq r_A\left( S-T\right) =r_A\left(
T-S\right).
\end{equation*}%
So, we get
\begin{equation}
-r_A\left(T-S\right) \leq r_A\left(T\right) -r_A\left( S\right) .  \label{e.2.4}
\end{equation}%
Therefore, we obtain the desired inequality \eqref{e.2.2} by combining \eqref{e.2.3} together with \eqref{e.2.4}.
\end{proof}

\begin{lemma}
\label{l.2.1}Let $\left(T_{n}\right) _{n\in \mathbb{N}}\subset \mathcal{B}_{A^{1/2}}(\mathcal{H})$ be a sequence of
$A$-bounded operators such that $T_{i}T_{j}=T_{j}T_{i}$ for any $i,j\in \mathbb{N}$. Then, for every $p\in \mathbb{N}^*$ we have%
\begin{equation}
r_A\left( \sum_{k=0}^{p}T_{k}\right) \leq \sum_{k=0}^{p}r_A\left(T_{k}\right) .
\label{e.2.1}
\end{equation}
\end{lemma}

\begin{proof}
We proceed by induction over $p$. If $p=1,$ then \eqref{e.2.1} is true by using Proposition \ref{22}. Now, assume that \eqref{e.2.1} holds for some $p>1.$ Since the
operators $T_j$ are commuting, then the operators $%
\sum_{k=1}^{p}T_{k}$ and $T_{p+1}$ commute. So, by using the induction hypothesis and applying Proposition \ref{22} we see that
\begin{align*}
r_A\left( \sum_{k=0}^{p+1}T_{k}\right)
& =r_A\left( \sum_{k=0}^{p}T_{k}+T_{p+1}\right)\\
&\leq r_A\left( \sum_{k=0}^{p}T_{k}\right)
+r_A\left(T_{p+1}\right) \\
& \leq \sum_{k=0}^{p}r_A(T_{k}) +r_A\left(T_{p+1}\right)=\sum_{k=0}^{p+1}r_A(T_{k}).
\end{align*}%
Hence, \eqref{e.2.1} is proved for any $p\in \mathbb{N}^*.$
\end{proof}

\begin{lemma}\label{l.2.3}
Let $\left(T_{n}\right) _{n\in \mathbb{N}}\subset \mathcal{B}_{A^{1/2}}(\mathcal{H})$ be a sequence of
$A$-bounded operators such that $T_{i}T_{j}=T_{j}T_{i}$ for any $i,j\in \mathbb{N}$. Let also $T\in \mathcal{B}_{A^{1/2}}(\mathcal{H})$. Then,
\begin{equation}\label{property}
\left(\lim_{n\rightarrow \infty }\|T_n-T\|_A=0\right)\Rightarrow \left(\lim_{n\rightarrow \infty }r_A\left(T_{n}\right) =r_A\left(T\right)\right).
\end{equation}
\end{lemma}

\begin{proof}
Notice first that, in general, by using \eqref{newrad} together with \eqref{semiiineq2} it can be checked that
\begin{equation}\label{e.0.2}
r_A(X)\leq \|X\|_A,\quad\forall\,X\in \mathcal{B}_{A^{1/2}}(\mathcal{H}).
\end{equation}
Since $T_n-T\in \mathcal{B}_{A^{1/2}}(\mathcal{H})$ for all $n\in \mathbb{N}$, then an application of \eqref{e.2.2} together with \eqref{e.0.2} gives
\begin{equation*}
\left\vert r_A\left(T_{n}\right) -r_A\left(T\right) \right\vert \leq r_A\left(
T_{n}-T\right) \leq \left\Vert T_{n}-T\right\Vert_A,
\end{equation*}%
for any $n\in \mathbb{N}$. Hence, the property \eqref{property} follows immediately.
\end{proof}

Now, we are in a position to state and prove the following result.

\begin{theorem}\label{t.2.1}
Let $T\in \mathcal{B}_{A^{1/2}}(\mathcal{H})$ be such that $A$ is an invertible operator. Let $f\left( z\right) =\sum_{n=0}^{\infty }c_{n}z^{n}$ be a
power series with complex coefficients and convergent on the open disk $%
D\left( 0,R\right) \subset \mathbb{C}$ with $R>0.$ If $\left\Vert T\right\Vert_A <R,$ then%
\begin{equation}
r_A\left[ f\left( T\right) \right] \leq f_{c}\left( r_A\left( T\right) \right),
\label{e.2.5}
\end{equation}
where $f_{c}\left( z\right) :=\sum_{k=0}^{\infty }\left\vert
c_{n}\right\vert z^{n}$.
\end{theorem}

\begin{proof}
Let $T\in \mathcal{B}_{A^{1/2}}(\mathcal{H})$ and consider the sequence $\{S_{n}\}$ in $\mathcal{B}_{A^{1/2}}(\mathcal{H})$ such that $S_{n}:=\sum_{k=0}^{n}c_{k}T^{k}$ for all $n\in \mathbb{N}^*$. For any $p,q\in \mathbb{N}^*$ with $p>q$ we have
 \begin{equation}\label{convergent}
\|S_p-S_q\|_A\leq \sum_{k=q+1}^p |c_k| \,\|T\|_A^k.
\end{equation}
Since $\left\Vert T\right\Vert_A <R$, then $\sum_{k\in \mathbb{N}} |c_k| \,\|T\|_A^k$ is convergent. So, it follows from \eqref{convergent} that $\{S_{n}\}$ is a Cauchy sequence in $\mathcal{B}_{A^{1/2}}(\mathcal{H})$. On the other hand, since $A$ is invertible, then it can checked that $(\mathcal{B}_{A^{1/2}}(\mathcal{H}),\|\cdot\|_A)$ is complete. This implies that
\begin{equation}\label{holds}
\lim_{n\to \infty}\|S_{n}-f\left(T\right)\|_A=0.
\end{equation}

Let $n\in \mathbb{N}^*$ and consider the operators $M_{k}:=c_{k}T^{k}$ for $k\in \left\{0,\cdots,n\right\}.$ Since $f\left( T\right) $ commutes with $M_k$ for all $k\in\{1,\cdots,n\}$, then we deduce that $f\left( T\right)S_n=S_nf(T)$. So, since \eqref{holds} holds, it follows from Lemma \ref{l.2.3} that
\begin{equation}\label{5ra}
\lim_{n\rightarrow \infty }r_A\left( S_n\right) =r_A\left[ f\left( T\right) \right].
\end{equation}%

On the other hand, it can observed that $M_{i}M_{j}=M_{j}M_{i}$ for any $i,j\in
\left\{ 0,\cdots,n\right\}.$ So, by applying Lemma \ref{l.2.1} we get
\begin{align*}
r_A\left( \sum_{k=0}^{n}c_{k}T^{k}\right)
& \leq \sum_{k=0}^{n}r_A\left(
c_{k}T^{k}\right) =\sum_{k=0}^{n}\left\vert c_{k}\right\vert r_A\left(
T^{k}\right).
\end{align*}%
This implies, by using Proposition \ref{22}, that
\begin{align}\label{e.2.6}
r_A\left( \sum_{k=0}^{n}c_{k}T^{k}\right)
& \leq \sum_{k=0}^{n}\left\vert c_{k}\right\vert \left[r_A\left( T\right)\right]^{k}.
\end{align}%
Since $r_A\left( T\right) \leq \left\Vert T\right\Vert_A <R,$ then $f_c(r_A(T))=\sum_{k=0}^{\infty }\left\vert c_{k}\right\vert
\left[r_A\left( T\right)\right]^{k}$ is convergent. Hence, by letting $n\rightarrow \infty$ in \eqref{e.2.6} and then using \eqref{5ra}, we obtain \eqref{e.2.5} as required.
\end{proof}

\begin{example}
Let $A$ be the diagonal positive operator on $\ell_{\mathbb{N}^*}^2(\mathbb{C})$ given by $Ae_{2n-1}=e_{2n-1}$ and $Ae_{2n}=2e_{2n}$ for all $n\geq 1$, where $(e_n)_{n\in \mathbb{N}^*}$ denotes the canonical basis of $\ell_{\mathbb{N}^*}^2(\mathbb{C})$. Clearly, $A$ is an invertible operator. On the other hand, it is well-known that
$$\frac{1}{%
1+z}=\sum_{n=0}^{\infty }\left( -1\right) ^{n}z^{n},\;\;\forall\,z\in D\left( 0,1\right) $$
and
$$\exp \left( z\right)=\sum_{n=0}^{\infty }\frac{z^{n}}{n!}\;\;\forall\,z\in \mathbb{C}.$$
Now, let $T\in \mathcal{B}_{A^{1/2}}\left(\mathcal{H}\right)$ be such that $\left\Vert T\right\Vert_A <1$. In view of Theorem \ref{t.2.1} we have
\begin{equation*}
r_A\left[ \left(I\pm T\right) ^{-1}\right] \leq \left[1-r_A\left(T\right)\right]^{-1}.
\end{equation*}%
Moreover, if $T\in\mathcal{B}_{A^{1/2}}\left(\mathcal{H}\right)$ then again by using Theorem \ref{t.2.1} we get
\begin{equation*}
r_A\left[\exp \left(T\right) \right] \leq \exp \left[r_A\left( T\right)\right].
\end{equation*}%
\end{example}


\begin{thebibliography}{10} 
\footnotesize
\bibitem{A.K.2} A. Abu-Omar, F. Kittaneh, Notes on some spectral radius and numerical radius inequalities, Studia Math. 227 (2015), no. 2, 97--109.

\bibitem{OK2} A. Abu-Omar and F. Kittaneh,  Numerical radius inequalities for $n \times n$ operator  matrices, Linear Algebra and its Application, 468 (2015), 18-26.

\bibitem{acg1}{M.L. Arias, G. Corach, M.C. Gonzalez,} {Partial isometries in semi-Hilbertian spaces,} Linear Algebra Appl. 428 (7) (2008) 1460-1475.

\bibitem{acg2} {M.L. Arias, G. Corach, M.C. Gonzalez,} {Metric properties of projections in semi-Hilbertian spaces,} Integral Equations and Operator Theory, 62 (2008), pp.11-28.

\bibitem{acg3} {M.L. Arias, G. Corach, M.C. Gonzalez,} {Lifting properties in operator ranges,} Acta Sci. Math. (Szeged) 75:3-4(2009), 635-653.


\bibitem{bakfeki02}{H. Baklouti, K.Feki,} {On joint spectral radius of commuting operators in Hilbert spaces,}  Linear Algebra Appl. 557 (2018), 455-463.

\bibitem{bakfeki01}{H. Baklouti, K.Feki, O.A.M. Sid Ahmed,} {Joint numerical ranges of operators in semi-Hilbertian spaces,}  Linear Algebra Appl. 555 (2018) 266-284.

\bibitem{bakfeki04}{H. Baklouti, K.Feki, O.A.M. Sid Ahmed,} {Joint normality of operators in semi-Hilbertian spaces,} Linear Multilinear Algebra (2019), \url{https://doi.org/10.1080/03081087.2019.1593925}, in press.

\bibitem{BBP} P. Bhunia, S. Bag and K. Paul, Numerical radius inequalities and its applications in estimation of zeros of polynomials, Linear Algebra Appl. 573 (2019) 166-177.

\bibitem{BP} P. Bhunia and K. Paul, Some improvements of numerical radius inequalities of operators and operator matrices, arXiv:1910.06775v1 [math.FA] (to appear in Linear and Multilinear Algebra).

\bibitem{bhunfekipaul}{P. Bhunia, K.Feki, K. Paul,} {$A$-Numerical radius orthogonality and parallelism of semi-Hilbertian space operators and their applications,}  arXiv:2001.04522v1 [math.FA] 13 Jan 2020.

\bibitem{branrov}{L. de Branges, J. Rovnyak,} {Square Summable Power Series,} Holt, Rinehert and Winston, New York, 1966.


\bibitem{doug}{R.G. Douglas,} {On majorization, factorization and range inclusion of operators in Hilbert space,} Proc. Amer. Math. Soc. 17 (1966) 413-416.

\bibitem{dfilot} S. S. Dragomir, Spectral Radius Inequalities for Functions of Operators Defined by Power Series, Filomat 30:10 (2016), 2847-2856.

\bibitem{fg}{M. Faghih-Ahmadi, F. Gorjizadeh,} {A-numerical radius of A-normal operators in semi-Hilbertian spaces,} Italian journal of pure and applied mathematics n. 36-2016 (73-78).

\bibitem{feki01} {K. Feki}, {Spectral radius of semi-Hilbertian space operators and its applications,} Annals of Functionnal Analysis (2020), \url{https://doi.org/10.1007/s43034-020-00064-y}.




\bibitem{Ki} F. Kittaneh, Spectral radius inequalities for Hilbert space operators, Proc. Amer. Math. Soc. 134 (2006), 385-390.



\bibitem{majsecesuci}{W. Majdak, N.A. Secelean, L. Suciu,} {Ergodic properties of operators in some semi-Hilbertian spaces,} Linear and Multilinear Algebra,
61:2, (2013) 139-159.

\bibitem{zamani2} M.S. Moslehian, Q. Xu, A. Zamani, Seminorm and numerical radius inequalities of operators in semi-Hilbertian spaces, Linear Algebra Appl. 591 (2020) 299-321.






\bibitem{zamani1}{A. Zamani,} {$A$-numerical radius inequalities for semi-Hilbertian space operators,} Linear Algebra Appl. 578(2019) 159-183.

 \end{thebibliography}
\end{document}